\documentclass[10pt]{amsart}
\usepackage{graphicx}
\usepackage{natbib}
\usepackage{hyperref}
\newtheorem{thm}{Theorem}
\newtheorem{cor}{Corollary}
\newtheorem{lem}{Lemma}

\newtheorem{exam}{Example}

\newtheorem*{Athm}{Theorem A}

\newcommand\Cq{\mathcal{C}}
\newcommand\z{\mathfrak{z}}
\newcommand{\norm}[1]{\left\Vert#1\right\Vert}

\newcommand{\R}{\mathbb R}

\newcommand{\eps}{\varepsilon}

\begin{document}

\title[Approximation of Linear Differential Equations with delay]{Approximation of Linear Differential Equations with variable delay on an Infinite Interval}%Transference of Uniform Stability of Linear Delayed Differential Equations to the Corresponding Difference Equations

\author[]{Daniel Sep\'ulveda}
\address{Escuela de Matem\'atica y Estad\'istica, Universidad Central de Chile}
\email{daniel.sep.oe@gmail.com}
\thanks{Thanks for the support of research project CIR 1418 at Universidad Central de Chile}
\subjclass[2010]{34K06, 34K07, 34K20, 34K28}

\maketitle

%%%%%%%%%%%%%%%%
%COMIENZO RESUMEN%%%
%%%%%%%%%%%%%%%%
\begin{abstract}
We study approximation of non-autonomous linear differential equations with variable delay over infinite intervals. 
We use piecewise constant argument to obtain a corresponding discrete difference equation. 
The study of numerical approximation over an unbounded interval is in correspondence with the problem 
of transference of qualitative properties between a continuous and the corresponding discrete dynamical systems. 
We use theory of differential equations with delay and theory of integral inequality to prove our main result.
We state sufficient conditions for: a) uniform approximation of solutions over an unbounded interval and b) transference of uniform asymptotic stability of non-autonomous linear differential equations with variable delay to the corresponding discrete difference equations. We improve and extended the results of \cite{cooke1994}.\\

\medskip
Keywords: Linear functional-differential equations, Theoretical approximation of solutions, Stability theory, Numerical approximation of solutions.\\

\subjclass[MSC2010]{ 34K06, 34K07, 34K20, 34K28}
\medskip

%Highlights:\\

%We study approximation of non-autonomous linear differential equations with variable delay over infinite %intervals. 
%We use piecewise constant argument to obtain a corresponding discrete difference equation.
%We state sufficient conditions for: a) uniform approximation of solutions over an unbounded interval and b) transference of uniform asymptotic stability of non-autonomous linear differential equations with variable delay to the corresponding discrete difference equations.

\end{abstract}

\section{Introduction}

We study approximation of non-autonomous functional differential equations over infinite intervals. 
Recently, several researchers have discretized systems of differential equations 
using piecewise constant argument, they obtained an Euler's approximation of the 
solution of original system, see \cite{mohamad2003,huang2006,abbas2013}.
These paper follow the ideas developed by \cite{gyori1991}, who was interested in the 
convergence of several approximation, by piecewise constant argument,
to the actual solution of a class of delay linear differential equations.
Actually, that work belongs to a series of papers about approximations of solutions
of differential equations with delay \citep[see][]{gyori1988,gyori1991,cooke1994,gyori1995,gyori2002,gyori2008}.
Numerical methods for differential equations with delay are well-known, see \cite{bellen2013}. 
Numerical analysis for the study of  stability of delay differential equations have been developed recently, see \cite{breda2015}. However, the study of numerical approximation of a solution  over an unbounded interval is in correspondence with the problem of transference of qualitative properties between a continuous dynamical system  and the  corresponding  discrete dynamical system.   %\cite{gyori1991} started the approximation of solution of differential equations with delay using piecewise constant argument. Since then several authors have been used or considered piecewise constant to obtain a discrete version of differential equations with delays.

In this work we study the approximation of the homogeneous non-autonomous linear differential equations with variable delay 
\begin{equation}\label{eq: devd}
x'(t)=-a(t)x(t-r(t)),
\end{equation}
where $a:[0,\infty)\to [0,\infty)$ and $r:[0,\infty)\to [0,q],$  with initial condition 
\begin{equation}\label{eq: icd}x(t)=\varphi(t),\quad -q\leq t\leq 0,\quad \varphi\in \Cq \equiv C([-q , 0], \R),\end{equation}
where  $q:=\sup_{t\in\R^+_0}\{r(t)\}$ is a positive real number.

Our aims are find conditions for:
\begin{enumerate}
\item  the uniform convergence of approximate solutions over 
$[0,\infty)$; and
\item  the transference of uniform stability  between solutions 
of linear differential equations with variable delay 
and its corresponding difference equation.
\end{enumerate}

The study of transference of uniform asymptotic stability between solutions of linear differential equations with delay and the corresponding discrete difference equation started with  \cite{cooke1994}.
\cite{cooke2000} studied the dynamic of the solutions of a singular difference equation with delay which can be interpreted as  an Euler discretizations of a singular differential equations with delay. They concluded that
numerical approximation of solutions of singularly perturbed delay differential equations maybe showing dynamics
which are irrelevant to the actual dynamics in these equations.
These type of difficulties, namely spurious fixed point, are  well-known for Runge-Kutta methods of numerical approximation for ordinary differential equations.
 %\cite{mohamad2000} proved Halanay-type inequalities for nonautonomous scalar systems with discrete and distributed
%delays. They use piecewise constant argument to obtain a discrete non-autonomous difference systems and then show  
%discrete time inequalities, which are analoguesof continuous time inequalities. Therefore, indirectly, they proved the transference of  differential inequality with delay to the corresponding difference equation.  
 
\cite{liz2002} proved a discrete version of Halanay's inequality and used it to obtain the transference of asymptotic stability between the solutions of
$$x'(t)=-ax(t)+b(t)f(t,x_t), \quad x_0=\phi,$$
and 
$$\frac{x_{n+1}-x_n}{h}=-ax_n+b(t_n)f(t_n,\phi_n).$$

\cite{gyori2002}  studied the transference of uniform asymptotic
stability between solutions of linear neutral differential equations with constant delay and the 
corresponding discrete difference equation.%, this way they extended the result of \cite{cooke1994}.

We use theory of differential equations with delay and a integral version of Halanay inequality to prove our main result. To our knowledge our result is the first, of this type, for non-autonomous linear differential equations with variable delay, we use Halanay's inequality to obtain our main theorem, in this way we complement, extend  and improve the results of \cite{cooke1994}.

The rest of the paper is organized as follows. In the next section, Halanay Inequality, some definitions
and preliminary results are presented. Section 3 is devoted to discretization using piecewise constant argument and approximation over compact interval. Section 4 we prove main results about transference of stability properties and 
approximation over non-compact interval. Finally, in Section 5 %\textbf{we show an illustrative example and} 
we examine some implications of our results.

%We assume that the zero solution of \eqref{eq: devd} is uniformly asymptotically stable, then we will prove that
%the zero solution of the corresponding discrete equation are uniformly asymptotically stable too.

%However, the behaviours of solutions of functional differential equations is different 
%to those of ordinary differential equations \cite{driver1977,gopalsamy2013}. For 
%example, if there exists $\alpha>0$ such that $a(t)\leq \alpha$ and $\alpha q\leq 3/2$, 
%then the zero solution of \eqref{eq: devd} is uniformly stable, but if $\alpha q >3/2$ 
%then there are equations with unbounded solutions 
%(see \cite{gopalsamy2013,yoneyama1987}).   
 
\section{Preliminaries and Halanay inequality}
\label{sec: HIP}

In this section we introduce usual notation and some definitions of theory of differential equations with delay. Suppose $q\geq 0$ is a given real number, $\R=(-\infty,\infty)$,  
%is an $n$-dimensional linear vector space over the real with norm $\norm{\cdot}$,
$C([a,b],\R)$ is the Banach space of continuous functions mapping the interval $[a,b]$ into $\R$ with the topology of uniform convergence. If $[a,b]=[-q,0]$ we let $\Cq=C([-q,0],\R )$ and designate the norm of an element $\varphi$ in $\Cq$ by $\norm{\varphi}=\sup_{-r\leq \theta\leq 0}|\varphi(\theta)|$.
If 
$$\sigma\in\R, A\geq 0,\mbox{ and }x\in C([-\sigma-q,\sigma+A],\R),$$
then for any $t\in[\sigma,\sigma+A],$ we let $x_t(\theta)=x(t+\theta), -q\leq \theta \leq 0.$ If $D$ is a subset of $\R\times\Cq, f:D\to\R$ is
 a given function and ``$\;'\;$'' represents the right-hand derivative, we say that the relation
\begin{equation}\label{eq: bt1.1}
x'(t)=f(t,x_t)
\end{equation}
is a \textit{differential equation with delay} or a \textit{retarded functional differential equation} on $D$. A function 
$x$ is said to be a \textit{solution} of Equation \eqref{eq: bt1.1} on $[\sigma-q,\sigma+A)$ if there are $\sigma \in\R$ and 
$A>0$ such that $x\in C([\sigma-q,\sigma+A),\R), (t,x_t)\in D$ and $x(t)$ satisfies Equation \eqref{eq: bt1.1} for 
$t\in[\sigma,\sigma+A)$. For given $\sigma\in\R$, $\phi\in \Cq$, we say $x(\sigma,\phi,f)$ is a solution of equation \eqref{eq: bt1.1} \textit{with initial value} $\phi$ at $\sigma$ if there is an $A>0$ such that $x(\sigma,\phi,f)$ is a solution
of equation \eqref{eq: bt1.1} on $[\sigma-q,\sigma +A]$ and $x_{\sigma}(\sigma,\phi,f)=\phi$.
The theory of differential equations with delay can be found in books \citep{driver1977} and \citep{hale1993}.

\subsection{Halanay Inequality}
\cite{halanay1966} proved an asymptotic formula for the solutions of a
differential inequality with delay, and applied it in the stability theory of linear systems with delay.
Since beginning of the twenty-first century several authors have been extended, improved  and called Halanay inequality to such inequality 
\citep[see][]{baker2000,mohamad2000,liz2000,pinto2000,liz2002,liz2005,ou2015}.

\label{subsec: AppeHal}
\begin{thm}\label{Halanay}
Let $t_0$ be a real number and $q$ be a non-negative number. If $v:[t_0-q,\infty)\to\R^+$ satisfies
\begin{equation}\label{Hal1} \frac{d}{dt}v(t)\leq -\alpha v(t)+\beta\left[\sup_{s\in[t-q,t]}v(s)\right];\quad t\geq t_0 
\end{equation} 
where $\alpha$ and $\beta$ are constants with $\alpha>\beta>0$, then 
\begin{equation}\label{Hal2} v(t)\leq \norm{v_{t_0}}_q e^{-\eta(t-t_0)}\mbox{ for } t\geq t_0, 
\end{equation} 
where $\eta$ is the unique positive solution of
\begin{eqnarray}\label{*}
\eta=\alpha-\beta e^{\eta q}.
\end{eqnarray}
\end{thm}
A statement of this theorem  can be found in \cite{driver1977}. Next we show  an integral version of Halanay inequality.

\begin{lem}\label{lem: 2.2} 
Consider 
\begin{equation}\label{2.2.1} v(t)\leq \norm{v_{t_0}}_qe^{-\sigma (t-t_0)}f(t)+\int_{t_0}^te^{-\sigma (t-s)}K\left[\sup_{u\in[s-q,s]}v(u)\right] ds,\quad t\geq t_0, \end{equation}
where $\sigma, q$ and $K$ are positive real numbers, and $f\in C(\R)$ is a positive non decreasing function.
If  $\sigma > K>0$, then there exist $\eta>0$ and $M >0$ such that
\begin{equation}
v(t)\leq  M\norm{v_{t_0}}_qf(t)e^{-\eta(t-t_0)},\quad t\geq t_0,
\end{equation}
where $\eta$ is the real solution of
\begin{eqnarray}\label{**}
\eta=-\sigma+K e^{\eta q}.
\end{eqnarray}
\end{lem}

\begin{proof}
We define $w(t):=\frac{v(t)}{\norm{v_{t_0}}_qf(t)}$, so we have
\begin{equation*}\label{2.2.4} w(t)\leq e^{-\sigma (t-t_0)}+\int_{t_0}^te^{-\sigma (t-s)}K\frac{\left[\sup_{u\in[s-q,s]}\norm{v_{t_0}}_qf(u)\right]}{\norm{v_{t_0}}_qf(t)}\left[\sup_{u\in[s-q,s]}w(u)\right]ds. \end{equation*}
Since the function $f$ is non decreasing we have
\begin{equation*}\label{2.2.5} w(t)\leq e^{-\sigma (t-t_0)}+\int_{t_0}^te^{-\sigma (t-s)}K\left[\sup_{u\in[s-q,s]}w(u)\right] ds. \end{equation*}
Now define
$$\mu(t):=\bigg\{ \begin{array}{ll} 1& \mbox{ for } t_0-q\leq t\leq t_0\\
e^{-\sigma (t-t_0)}+\int_{t_0}^te^{-\sigma(t-s)}K \left[\sup_{u\in[s-q,s]}w(u)\right]ds& \mbox{ for } t_0\leq t.\end{array}  $$
Then $\mu$ is continuous and nonegative, and $w(t)\leq \mu(t)$ for $t_0-q\leq t$. Moreover, for $t_0\leq t$ 
$$\mu'(t)\leq-\sigma \mu (t)+K\left[\sup_{u\in[t-q,t]}\mu(u)\right].$$
Since $\sigma > K>0$, by Halanay's inequality, there exists $\eta>0$ such that
$$\mu(t)\leq e^{-\eta(t-t_0)},\quad t\geq t_0.$$
Therefore
$$w(t)\leq\mu(t)\leq e^{-\eta(t-t_0)},\quad t\geq t_0.$$
and 
$$v(t)\leq  \norm{v_{t_0}}_qf(t)e^{-\eta(t-t_0)},\quad t\geq t_0.$$
\end{proof}

\section{Discretization by piecewise constant argument and approximation over compact interval}
\label{sec: DPCA}
%We consider the homogeneous non-autonomous linear differential equations with variable delay 
%\begin{equation}\label{eq: devd}
%x'(t)=-a(t)x(t-r(t)),
%\end{equation}
%where $a:[0,\infty)\to [0,\infty)$ and $r:[0,\infty)\to [0,q],$  with initial condition 
%\begin{equation}\label{eq: icd}x(t)=\varphi(t),\quad -q\leq t\leq 0,\quad \varphi\in \Cq \equiv C([-q , 0], \R),\end{equation}
%where  $q:=\sup_{t\in\R^+_0}\{r(t)\}$ is a positive real number.
The set of DEPCA corresponding  to equation \eqref{eq: devd} is: 
\begin{equation}\label{eq: devda}z_h'(t)=a(t)z_h\left(\Big[t-\Big[r\left([t]_h\right)\Big]_h\Big]_h\right),\end{equation}
where $[\cdot]_h=[\frac{\cdot}{h}]h$ with $[\cdot]$ the usual greatest integer function.
The initial condition for differential equation \eqref{eq: devda} is
\begin{equation}\label{eq: ica}z_h(nh)=\varphi(nh),\quad n=-k,\cdots, 0, \end{equation}
where $h$ is a positive number in the interval $(0,q]$. In fact we can consider 
$h=\frac{q}{k}$ where $k\geq 1$ is an integer.
By a solution of  \eqref{eq: devda}-\eqref{eq: ica} we mean a function $z_h$ defined on 
$\{ih: i=-k,\cdots, 0\}$ by \eqref{eq: ica}, which satisfy the following properties on $\R^+$: 
\begin{enumerate}
  \item[i)] The function $z_h$ is continuous on $\R^+$, 
  \item[ii)] the derivative $z_h'(t)$ exists at each point $t\in \R^+$ with the possible exception of the points
$ih (i=0,1,2,\cdots )$ where finite one-sided derivatives exist, and
  \item[iii)] the function $z_h$ satisfies \eqref{eq: devda} on each interval $I_{(j,h)}:=[jh,(j+1)h)$
for $j=0,1,2,\cdots$.
\end{enumerate}
Note that for every positive $h$ close to zero, it is expected that solutions of \eqref{eq: devda}-\eqref{eq: ica} has similar
qualitative features to the solutions of \eqref{eq: devd}-\eqref{eq: icd},
since $[t]_h\to t$ uniformly on $\R$, as $h\to 0$. But this can be false, even for Euler's method in the setting of singular functional differential equations, see \cite{cooke2000}.\\ 
If we denote $$\gamma_h(t,r):=\Big[\frac{t}{h}-\Big[\frac{r([\frac{t}{h}]h)}{h}\Big]\Big]h,$$ then equation \eqref{eq: devda} 
can be write like
$$z_h'(t)=a(t)z_h(\gamma_h(t,r)).$$
For $ t \in I_{(i,h)}$ the function $r([\frac{t}{h}]h)$  take just one value, therefore the function $\Big[\frac{r([\frac{t}{h}]h)}{h}\Big]:=k_i$ is a fixed integer for $ t \in I_{(i,h)}$. It follows that 
\begin{eqnarray*}
\Big[\frac{t}{h}-\Big[\frac{r([\frac{t}{h}]h)}{h}\Big]\Big]& = & \Big[\frac{t}{h}-k_i\Big]\\
% & = & \Big[\frac{t-ih+ih}{h}-k_i\Big]\\
  & = & \Big[\frac{t-ih}{h}+i-k_i\Big],
\end{eqnarray*}
since $t\in I_{(i,h)}$ we have   $0\leq \frac{t-ih}{h}\leq 1$, so
\begin{eqnarray*}
\Big[\frac{t}{h}-\Big[\frac{r([\frac{t}{h}]h)}{h}\Big]\Big]& = & i-k_i.
\end{eqnarray*}
It is follows that
$$\gamma_h(t,r)=h(i-k_i),\mbox{ for }t\in I_{(i,h)}.$$
Therefore \eqref{eq: devda} is equivalent to 
\begin{equation}\label{eq: devda.a}z_h'(t)=a(t)z_h(h(i-k_i)),\quad t\in I_{(i,h)},\; i\geq 0.\end{equation}
Note that, for $ih\leq t\leq (i+1)h$,  we integrate \eqref{eq: devda.a} and obtain 
\begin{eqnarray*}
%\int_{ih}^{t}z_h'(s)ds&=&\int_{ih}^{t}a(s)z_h(h(i-k_i))ds\\
z_h(t)&=&z_h(ih)+\int_{ih}^{t}a(s)ds z_h(h(i-k_i)).
\end{eqnarray*}
Making $t\to (i+1)h^-$, from the continuity of $z_{h}$, we obtain
\begin{eqnarray*}
z_h((i+1)h)ds&=&z_h(ih)+\int_{ih}^{(i+1)h}a(s)ds z_h(h(i-k_i)).
\end{eqnarray*}
Therefore the sequence $\z_{h}(i):=z_{h}(ih)$ satisfy the linear  difference equation with variable delay
\begin{equation}\label{eq: ddvd}\z_{h}(n+1)=\z_{h}(n)+\int_{nh}^{(n+1)h}a(s)ds \z_{h}(n-k_n), 
\end{equation}
with initial conditions
\begin{equation}\label{eq: icdd} \z_{h}(n)=\varphi(nh), \quad n=0, 1,\cdots,\quad -q\leq -nh\leq 0.
\end{equation}
Note that \eqref{eq: ddvd}-\eqref{eq: icdd} is a discretization of the differential equations with variable delay \eqref{eq: devd}-\eqref{eq: icd} that 
coincides with Euler's approximation method for autonomous differential equations with variable delay.  
From the recurrence relation \eqref{eq: ddvd} and initial conditions, we have
\begin{eqnarray*}
\z_{h}(0)& = &\varphi(0)\\
\z_{h}(1)& = & \z_{h}(0)+\int_{0}^{h}a(s)ds\z_{h}(0-k_0),\\
\z_{h}(2)& = & \z_{h}(1)+\int_{h}^{2h}a(s)ds\z_{h}(1-k_1)\\
        & = & \z_{h}(0)+\int_{0}^{h}a(s)ds\z_{h}(0-k_0)+\int_{h}^{2h}a(s)ds\z_{h}(1-k_1).
\end{eqnarray*}
Therefore the sequence $\z_{h}$ solution of \eqref{eq: ddvd}-\eqref{eq: icdd} is well-defined, and satisfy
\begin{equation}\label{2.4.8}\z_{h}(n)=\varphi(0)+\sum_{i=0}^{n-1}\int_{ih}^{(i+1)h}a(s)ds \z_h(i-k_i), \quad k\geq 0.\end{equation}
From (\ref{2.4.8}), it follows that the solution of \eqref{eq: devda}-\eqref{eq: ica} for $t\geq 0$ can be written
\begin{equation}\label{2.4.10}z_{h}(t)=\z_{h}(n)+\int_{nh}^ta(s)ds \z_{h}\left(n-k_n \right),\end{equation}
or
$$z_{h}(t)=\varphi(0)+\sum_{i=0}^{n-1}\int_{ih}^{(i+1)h}a(s)ds \z_h(i-k_i)+\int_{nh}^ta(s)ds \z_{h}\left(n-k_n \right),$$
where $n=n(t)$ is such that $nh\leq t < (n+1)h$. Thus, we have proved
\begin{thm}\label{Teo2.1}The initial value problem \eqref{eq: devda}-\eqref{eq: ica} has  a unique solution in the form
\begin{equation}\label{2.4.9}
z_{h}(t)=\varphi(0)+\sum_{i=0}^{n-1}\int_{ih}^{(i+1)h}a(s)ds \z_h(i-k_i)+\int_{nh}^ta(s)ds \z_{h}\left(n-k_n \right)
\end{equation}
for $t\geq 0$ and the sequence $\z_{h}(\cdot)$ satisfies the nonlinear difference equations \eqref{eq: ddvd} 
with initial conditions \eqref{eq: icdd}.
\end{thm}

\subsection{Aproximation over compact interval}
\label{subsec: App}
In this subsection we address the problem of approximation of solutions of  
initial value problem \eqref{eq: devd}-\eqref{eq: icd} over compact interval. 
We follows some ideas of \cite{gyori1991} to obtain.  
\begin{thm}\label{Teo2.1.1}If $r:[0,\infty)\to [0,q]$ is a continuous function then,  for any  $\varphi\in \Cq$, the solutions $x(\varphi)(t)$ and $z_{h}(\varphi)(t)$ 
of the initial value problems \eqref{eq: devd}-\eqref{eq: icd} and \eqref{eq: devda}-\eqref{eq: ica}, respectively, satisfy the following relations for all $T>0$
\begin{equation}\label{2.1.1} \lim_{h \to 0} \max_{0\leq t\leq  T}\norm{x(\varphi)(t)-z_{h}(\varphi)(t)}=0,
\end{equation}
namely
\begin{equation}\label{2.1.7}\max_{0\leq t\leq T}\norm{x(\varphi)(t)-z_{h}(\varphi)(t)}\leq \left[ e^{\int_0^Ta(s)ds }\int_0^Ta(s)ds\right] \; w_x\left(w_r(h;T)+2h;T\right).
\end{equation}
where $w_r(h;T)$  and $w_x(h;T)$ are defined by
\begin{eqnarray*}
w_r(h;T)&=&\max \left\{|r(t_2)-r(t_1)|:\; 0\leq t_1,t_2\leq t, \; |t_2-t_1|\leq h \right\},\\
w_x\left(w_r(h;t)+2h;t\right)&=&\max \left\{|x(t_2)-x(t_1)|:\; -q\leq t_1,t_2\leq t, \; |t_2-t_1|\leq 2h+w_r(h,t) \right\}.
\end{eqnarray*}
\end{thm}

\begin{proof}
Consider the solutions $x(t)=x(\varphi)(t)$ and $z_{h}(t)=z_{h}(\varphi)(t)$ of initial value problems \eqref{eq: devd}-\eqref{eq: icd} and \eqref{eq: devda}-\eqref{eq: ica}, respectively.
 Then from \eqref{eq: devd} and \eqref{eq: devda} we find
$$x'(t)-z'_{h}(t)=-a(t)\left[x(t-r(t))-z_{h}(\gamma_{h}(t,r))\right],$$
for all $t\geq 0.$ Thus the function $\varepsilon_{h}(t)=x(t)-z_{h}(t)$ satisfies
$$\varepsilon'_{h}(t)=-a(t)\varepsilon_{h}(\gamma_{h}(t,r))-a(t)\left[x(t-r(t))-x(\gamma_{h}(t,r))\right] ,$$
for all $t\geq 0$ with  $\varepsilon_{h}(0)=0$. We integrate over $[0,t]$ and obtain
\begin{eqnarray*}
|\varepsilon_{h}(t) |&\leq & \int_0^ta(s)|\varepsilon_{h}(\gamma_{h}(s,r))|ds+\int_0^ta(s)\left|x(s-r(s))-x(\gamma_{h}(s,r))\right|ds\\
 &\leq & \int_0^ta(s)|\varepsilon_{h}(s,r(s))|ds+f_h(t),
\end{eqnarray*}
where
%\begin{equation}\label{def: fh} 
%f_h(t):=\int_0^ta(s)\left|x(s-r(s))-x(\gamma_{h}(s,r))\right|ds,\quad t\geq 0.
%\end{equation}
$$f_h(t):=\int_0^ta(s)\left|x(s-r(s))-x(\gamma_{h}(s,r))\right|ds,\quad t\geq 0.$$
On the other hand, 
$$\gamma_{h}(s,r)\leq s\mbox{ for all }s\geq 0,$$
and from  initial conditions we have that
$$|\varepsilon_{h}(\gamma_{h}(s,r))|=|\varphi(\gamma_{h}(s,r))-x_{h}(\gamma_{h}(s,r))|=0,$$
for all $s\geq 0$ such that $\gamma_{h}(s-r)\leq 0.$
Therefore we find that the function $\xi(t)=\max_{0\leq s\leq t} |\varepsilon_{h}(s)|$ satisfies the inequality
$$\xi(t)\leq \int_0^ta(s)\xi(\gamma_{h}(s,r))ds+f_h(t)\leq \int_0^ta(s)\xi(s))ds+f_h(t), \quad t\geq 0,$$
where we used that the integral term and $f_h(t)$ are monotone increasing functions. 
By Gronwall-Bellman inequality we find 
\begin{equation}\label{2.1.2}
\xi(t)\leq f_h(t) e^{\int_0^ta(s)ds},\quad t\in [0,T].
\end{equation}
Now, we note that $|t-r(t)-\gamma_h(t,r)|=|t-r(t)-(i-k_i)h|$ where $i=\left[\frac{t}{h}\right]$ and $k_i=\Big[\frac{r([\frac{t}{h}]h)}{h}\Big]$, so
\begin{eqnarray}
|t-r(t)-\gamma_h(t,r)|& \leq &|t-ih|+|r(t)-k_ih| \nonumber \\
        & = &  \left|t-\left[\frac{t}{h}\right]h \right|+\left|r(t)-\Big[\frac{r([\frac{t}{h}]h)}{h}\Big]h\right| \nonumber \\
        & \leq & h+\left|r(t)-r([\frac{t}{h}]h)\right|+\left|r([\frac{t}{h}]h)-\Big[\frac{r([\frac{t}{h}]h)}{h}\Big]h\right| \nonumber \\
        & \leq & h+\left|r(t)-r([\frac{t}{h}]h)\right|+h \nonumber \\
        & \leq & 2h+w_r(h;t),\label{eq: estinterval}
\end{eqnarray}
where  $w_r(h;t)=\max \left\{|r(t_2)-r(t_1)|:\; 0\leq t_1,t_2\leq t, \; |t_2-t_1|\leq h \right\}.$
Note that for uniformly continuous function $r$, $w_r(h,t)$ tends to zero as $h$ tends to $0$.
Set
$$ w_x\left(w_r(h;t)+2h;t\right)=\max \left\{|x(t_2)-x(t_1)|:\; -q\leq t_1,t_2\leq t, \; |t_2-t_1|\leq 2h+w_r(h,t) \right\}.$$
Then from (\ref{eq: estinterval}) it is follows that
$$ |x(s-r(s)) -x(\gamma_{h}(s,r))|\leq w_x\left(w_r(h;t)+2h;t\right),$$
for all $0\leq s\leq t$ and for all $r.$ Also,
$$f_h(t)\leq \int_0^ta(s)ds\; w_x(w_r(h;t)+2h;t), \quad t\geq 0,$$
and clearly (\ref{2.1.2}) yields
\begin{equation}\label{2.1.3}
\xi(t)\leq e^{\int_0^ta(s)ds }\int_0^ta(s)ds\; w_x\left(w_r(h;T)+2h;T\right),\quad t\in [0,T].
\end{equation}
Since  \eqref{2.1.3} and $|\varepsilon_{h}(t)|=|x(t)-z_{h}(t)|=|x(\varphi)(t)-z_{h}(\varphi)(t)|\leq \xi(t)$ we obtain \eqref{2.1.7} for all $h=\frac{q}{k}>0$ and $t\in [0,T]$.
%$$|x(\varphi)(t)-z_{h}(\varphi)(t)|\leq e^{\int_0^ta(s)ds }\int_0^ta(s)ds$$
%\begin{equation}\label{2.1.6}|x(\varphi)(t)-z_{h}(\varphi)(t)|\leq e^{\int_0^ta(s)ds }\int_0^ta(s)ds
%\end{equation}
So, for all $T>0$
\begin{equation*}\max_{0\leq t\leq T}|x(\varphi)(t)-z_{h}(\varphi)(t)|\leq \left[ e^{\int_0^Ta(s)ds }\int_0^Ta(s)ds \right] \; w_x\left(w_r(h;T)+2h;T\right)\to 0,
\end{equation*}
as $h\to 0,$ from the uniform continuity of the functions $x$ and $r$ on $[0,T]$.
\end{proof}

%\begin{cor}\label{Cor2.1.1} For any  $\varphi\in \Cq$ the solution $x(\varphi)(t)$ and $\z_{h}(\varphi)(t)$ 
%of the initial value problems \eqref{eq: devd}-\eqref{eq: icd} and  \eqref{eq: ddvd}-\eqref{eq: icdd}, respectively, satisfy the following relations for all $T>0$
%\begin{equation}\label{2.1.1a} \lim_{h \to 0} \max_{0\leq t\leq  T}\norm{x(\varphi)(t)-\z_{h}(\varphi)(t)}=0,
%\end{equation}
%namely
%\begin{equation}\label{2.1.7a}\max_{0\leq t\leq T}\norm{x(\varphi)(t)-\z_{h}(\varphi)(t)}\leq %e^{\int_0^Ta(s)ds }\int_0^Ta(s)ds \; w(x; T,h).
%\end{equation}
%\end{cor}

%\begin{proof}
%Consider the solutions $x(t)=x(\varphi)(t)$ and $z_{h}(t)=z_{h}(\varphi)(t)$ of initial value %problems \eqref{eq: devd}-\eqref{eq: icd} and \eqref{eq: ddvd}-\eqref{eq: icdd}, respectively.
% Then from \eqref{eq: devd} and \eqref{eq: devda} we find
%$$ |x(t)-\z_{h}(t)|\leq|x(t)-z_{h}(t)|+|z_{h}(t)-\z_{h}(t)|$$
%\end{proof}

\section{Transference of asymptotic stability}\label{sec: Tsta}
In this section we obtain a sufficient conditions for : a) uniform approximation of solutions over an unbounded interval and b) transference of uniform asymptotic stability  of the zero solution of non-autonomous linear differential equations with variable delay  of \eqref{eq: devd} to the  zero solution of the corresponding discrete difference equations  \eqref{eq: ddvd}.
\begin{enumerate}
 \item[(A1)] The zero solution of \eqref{eq: devd} is uniformly asymptotically stable,
 \item[(A2)] the function $a(t)$ is bounded, namely 
\begin{equation}\label{eq: aa}
a_0=\sup_{t\geq 0}|a(t)|<\infty.
\end{equation} 
  \item[(A3)] The function $r(t)$ is uniformly continuous on $[0,\infty).$
\end{enumerate}
Next we will obtain an estimate for the distance between the solutions 
of initial value problems \eqref{eq: devd}-\eqref{eq: icd} and \eqref{eq: devda}-\eqref{eq: ica} on $[0,\infty).$
\begin{thm}\label{teo: 1}
If the assumptions (A1), (A2) and (A3) holds. Then for $h$ small enough, and for every $\varphi \in \Cq$ the solutions $x(\varphi)(t)$ and $z_h(\varphi)(t)$ 
of the linear differential equations with delay \eqref{eq: devd} and \eqref{eq: devda}, respectively, satisfy
$$|x(\varphi)(t)-z_h(\varphi)(t)|\leq \left[K\norm{E_{h_{t_0}}}_q+K_1(h)M_1\norm{\varphi}_q t \right]e^{-\eta(t-t_0) },\quad t\geq t_0,$$
where $\eta>0,\; t_0=3q+w_{r}(q)$  and $K_1(h)=a_0^2\left[2h+w_{r}(h)\right]K,$ with $a_0$ defined in \eqref{eq: aa} and
$$w_r(\ell)=\max \left\{|r(t_2)-r(t_1)|:\; 0\leq t_1,t_2; \; |t_2-t_1|\leq \ell \right\},$$
and $K_1(h)\to 0$ as $h\to 0$.
\end{thm}

\begin{proof} 

Consider the solutions $x(t)=x(\varphi)(t)$ and $z_h(t)=z_h(\varphi)(t)$ of differential equations with delay \eqref{eq: devd} and \eqref{eq: devda}, respectively.
We define the error function $E_h(\cdot):=x(\cdot)-z_h(\cdot)$,  it follows that
$$E_h'(t)=-a(t)\left[x(t-r(t))-z_h(\gamma_h(t,r)) \right],$$
for all $t\geq 0.$ Adding and substracting $a(t)z(t-r(t))$ we obtain
$$E_h'(t)=-a(t)E_h(t-r(t))-a(t)\left[z_h(t-r(t))-z_h(\gamma_h(t,r))\right],$$
%or
%$$x'(t)-z'_h(t)=-a(t)\left[x(t-r(t))-z_h(t-r(t)) \right]+a(t)\left[z_h(t-r(t))-z_h(\gamma_h(t,r))\right]-a(t)\left[z_h(t-r(t))-z_h(\gamma_h(t,r))\right],$$
and, by fundamental theorem of calculus, we have
$$E_h'(t)=-a(t)E_h(t-r(t))-a(t)\int_{\gamma_h(t,r)}^{t-r(t)}z'_h(\xi)d\xi.$$
Now, from \eqref{eq: devda}, we obtain  
$$E_h'(t)=-a(t)E_h(t-r(t))-a(t)\int_{\gamma_h(t,r)}^{t-r(t)}a(\xi)z_h(\gamma_h(\xi,r))d\xi,$$
it is follows that
\begin{equation}\label{eq: errdevd}
E_h'(t)=-a(t)E_h(t-r(t))+a(t)\int_{\gamma_h(t,r)}^{t-r(t)}a(\xi) E_h(\gamma_h(\xi,r))d\xi+g_h(t),\nonumber
\end{equation}
where
\begin{equation}\label{eq: gdevd}
g_h(t):=-a(t)\int_{\gamma_h(t,r)}^{t-r(t)}a(\xi)x(\gamma_h(\xi,r))d\xi.
\end{equation}
Since variation-of-constants formula, \citep[see][pp. 334]{driver1977}, we have
$$E_h(t)=U(t;\tau,E_{h_{\tau}})+\int_{\tau}^tU\left(t;s,\left[a(s)\left(\int_{\gamma_h(s,r)}^{s-r(s)}a(\xi)E_h(\gamma_h(\xi,r))d\xi\right) +g_h(s)\right]u\right)ds,$$
where $U(t;\tau,E_{h_{\tau}})$ is the unique solution of 
Equation \eqref{eq: devd} with initial value $E_{h_{\tau}}$ at $\tau$,
and $u$ is the unit step function $u:[-q,0]\to\R$ defined by
\begin{equation}\label{eq: usf}
u(t)=\bigg\{\begin{array}{ll}0,\mbox{ for }-q\leq t<0,\\1,\mbox{ for } t=0.
\end{array}
\end{equation}

%evolution family  of the linear differential equation with variable delay  \eqref{eq: devd}
Thus $|E_h(t)|$ for all $t\geq t_0$ satisfies
\begin{equation}\label{eq: esti}
|E_h(t)|\leq |U(t;t_0,E_{h_{t_0}})|+\int_{t_0}^t\left|U\left(t;s,\left[a(s)\left(\int_{\gamma_h(s,r)}^{s-r(s)}a(\xi)E_h(\gamma_h(\xi,r))d\xi\right) +g_h(s)\right]u\right)\right|ds.
\end{equation}

%From \eqref{eq: aa} we obtain
%\begin{equation}\label{eq: esti}
%|E_h(t)|\leq|U(t,t_0)E_h(t_0)|+ a_0^2\int_{t_0}^t|U(t,s)|\left|\int_{\gamma_h(s,r)}^{s-r(s)}E_h(\gamma_h(\xi,r))d\xi\right|ds+ a_0\int_{t_0}^t|U(t,s)||g_h(s)|ds.
%\end{equation}
Since we assume that zero solution of \eqref{eq: devd} is uniformly asymptotically stable, there are 
constants $\sigma>0$ and $K>0$, \citep[see][pp. 185]{hale1993}, such that for each $\phi\in \Cq$ we have
\begin{equation}\label{eq: es}
|U(t;s,\phi)|\leq K\norm{\phi}_qe^{-\sigma(t-s)}, \quad t\geq s.
\end{equation}
Moreover, there exists a constant $M_0$ such that
\begin{equation}\label{eq: esx}
|x(t)|\leq M_0\norm{\varphi}_qe^{-\sigma t}, \quad t\geq 0.
\end{equation}
In order to use \eqref{eq: esx} to estimate $g_h(t)$ we need find a positive real number
$t_0$ such that for $t\geq t_0$ then 
$$0\leq \gamma_h(\xi,t),\mbox{ whenever }\xi\geq \gamma_h(t,r).$$
We recall \eqref{eq: estinterval}, i.e., $|t-r(t)-\gamma_h(t,r)|=|\gamma_h(t,r)-t+r(t)|\leq 2h+w_{r}(h)$, it is follows
$$t-r(t)-w_{r}(h)-2h\leq \gamma_h(t,r)\leq t-r(t)+w_{r}(h)+2h.$$
Since $h\in (0,q]$ and $r$ is uniformly continuous on $[0,\infty)$ it  follows 
\begin{equation}\label{eq: est 24}
t-3q-w_r(q)\leq t-r(t)-w_r(q)-2q\leq t-r(t)-w_r(h)-2h\leq \gamma_h(t,r).
\end{equation}
Now we use \eqref{eq: esx} and \eqref{eq: aa} in \eqref{eq: gdevd} to estimate $g_h(t)$ for $t\geq t_0:=3q+w_r(q)$ and obtain:
\begin{eqnarray}
|g_h(t)|& \leq &a_0^2\int_{\gamma_h(t,r)}^{t-r(t)}|x(\gamma_h(\xi,r))|d\xi\nonumber \\
        & \leq &a_0^2\int_{\gamma_h(t,r)}^{t-r(t)}M_0\norm{\varphi}_qe^{-\sigma \gamma_h(\xi,r))}d\xi\nonumber \\
        & = &a_0^2\int_{\gamma_h(t,r)}^{t-r(t)}M_0\norm{\varphi}_qe^{-\sigma\xi} e^{\sigma[\xi- \gamma_h(\xi,r))]}d\xi\nonumber,
\end{eqnarray}
since \eqref{eq: estinterval} and uniform continuity of $r$ we have
\begin{eqnarray}
|g_h(t)|& \leq &a_0^2\int_{\gamma_h(t,r)}^{t-r(t)}M_0\norm{\varphi}_qe^{-\sigma\xi} e^{\sigma[w_{r}(h)+2h]}d\xi\nonumber \\
                & \leq &e^{-\sigma \gamma_h(t,r))}a_0^2\int_{\gamma_h(t,r)}^{t-r(t)}M_0\norm{\varphi}_qe^{\sigma[w_{r}(h)+2h]}d\xi. \label{eq: estgh}
\end{eqnarray}
Next, using \eqref{eq: estinterval},  we estimate $s-\gamma(s,r)$
\begin{eqnarray}
s-\gamma(s,r)& = & s-r(s)-\gamma(s,r) +r(s) \nonumber \\
                & \leq & q+2h+w_r(h)\nonumber\\
                & \leq & 3q+w_r(q).\label{eq: estima}
\end{eqnarray}
Since $t\geq t_0 := 3q+w_r(q)$, \eqref{eq: estinterval}, \eqref{eq: est 24} and \eqref{eq: estima} inequality \eqref{eq: estgh} become into
\begin{eqnarray}
|g_h(t)|& \leq & e^{-\sigma t}  a_0^2M_1\norm{\varphi}_q\left[2h+w_r(h)\right], \label{eq: estghb}
\end{eqnarray}
where $$M_1:=M_0 e^{\sigma[5q+2w_r(q)]}.$$%\geq M_0 e^{\sigma[3q+2w_r(q)+2h]}
Using estimations \eqref{eq: es} and \eqref{eq: estghb} for $t_0:=3q+w_r(q)$, in \eqref{eq: esti}  we obtain for $t\geq t_0$
\begin{eqnarray*}
|E_h(t)|&\leq &|U(t;t_0,E_{h_{t_0}})|+\int_{t_0}^t\left|U\left(t;s,\left[a(s)\left(\int_{\gamma_h(s,r)}^{s-r(s)}a(\xi)E_h(\gamma_h(\xi,r))d\xi\right) +g_h(s)\right]u\right)\right|ds\\
&\leq & K\norm{E_{h_{t_0}}}_qe^{-\sigma(t-t_0)} +\int_{t_0}^tK\norm{\left[a(s)\left(\int_{\gamma_h(s,r)}^{s-r(s)}a(\xi)E_h(\gamma_h(\xi,r))d\xi\right) +g_h(s)\right]u}_qe^{-\sigma(t-s)}ds\\
&\leq & K\norm{E_{h_{t_0}}}_qe^{-\sigma(t-t_0)} +\int_{t_0}^tK \left(\left| a(s)\left(\int_{\gamma_h(s,r)}^{s-r(s)}a(\xi)E_h(\gamma_h(\xi,r))d\xi\right) \right| +|g_h(s)| \right)e^{-\sigma(t-s)}ds\\
&\leq & K\norm{E_{h_{t_0}}}_qe^{-\sigma(t-t_0)} + a_0^2\int_{t_0}^tK \left|\int_{\gamma_h(s,r)}^{s-r(s)}E_h(\gamma_h(\xi,r))d\xi\right|e^{-\sigma(t-s)}ds\\
& & +\int_{t_0}^tK e^{-\sigma s}  a_0^2M_1\norm{\varphi}_q\left[2h+w_r(h)\right]e^{-\sigma(t-s)}ds\\
&\leq & K\norm{E_{h_{t_0}}}_qe^{-\sigma(t-t_0)} +a_0^2\int_{t_0}^tK\left( \sup_{\gamma_h(s,r)\leq \varsigma\leq s-r(s)}|E_h(\varsigma)|\right)\int_{\gamma_h(s,r)}^{s-r(s)}d\xi e^{-\sigma(t-s)}ds\\
& & +a_0^2KM_1\norm{\varphi}_q\left[2h+w_r(h)\right]e^{-\sigma t}t\\
&= & K\norm{E_{h_{t_0}}}_qe^{-\sigma(t-t_0)} +a_0^2\int_{t_0}^tK \left(\sup_{\gamma_h(s,r)\leq \varsigma\leq s-r(s)}|E_h(\varsigma)|\right)\left(s-r(s)-\gamma_h(s,r)\right)e^{-\sigma(t-s)}ds\\
& & +a_0^2KM_1\norm{\varphi}_q\left[2h+w_r(h)\right]e^{-\sigma t}t\\
&\leq & K\norm{E_{h_{t_0}}}_qe^{-\sigma(t-t_0)} +a_0^2\left[2h+w_r(h)\right]\int_{t_0}^tK \left(\sup_{s-3q-w_r(h)\leq \varsigma\leq s}|E_h(\varsigma)|\right) e^{-\sigma(t-s)}ds\\
& & +a_0^2KM_1\norm{\varphi}_q\left[2h+w_r(h)\right]e^{-\sigma t}t\\
&= & e^{-\sigma(t-t_0)}\left[ K\norm{E_{h_{t_0}}}_q +K_1(h)M_1\norm{\varphi}_qe^{\sigma t_0}t\right]+  K_1(h)\int_{t_0}^t  e^{-\sigma(t-s)}\sup_{s-t_0\leq \varsigma\leq s}|E_h(\varsigma)| ds,
\end{eqnarray*}
where $K_1(h):=a_0^2\left[2h+w(r;h)\right]K.$ If $h$ is small enough such that:
 $$\sigma>K_1(h),$$
then by Halanay type inequality (Lemma \ref{lem: 2.2}) there exists $\eta>0$  such that
\begin{equation}\label{eq: 2.2.3}
|E_h(t)|\leq \left[ K\norm{E_{h_{t_0}}}_q +K_1(h)M_1\norm{\varphi}_qe^{\sigma t_0}t\right]e^{-\eta(t-t_0) },\quad t\geq t_0,
\end{equation}
where $\eta$ is the positive solution of $$\eta=\sigma-K_1(h)e^{\eta q}.$$
\end{proof}
Thus we have proved the fundamental theorem of this chapter. There are similar 
results  to our results,
however the novelty of our theorem lies in the considered delayed differential 
equation and the technique used in the proof.
In \cite{cooke1994} and \cite{gyori2002} used Gronwall-Bellman 
inequality to obtain an estimating exponential decay. 
Using Gronwall-Bellman inequality we shall obtain: 
$$|E_h(t)|\leq \left[ K\norm{E_{h_{t_0}}}_q +K_1(h)M_1\norm{\varphi}_qe^{\sigma t_0}t\right]e^{-\sigma_0(t-t_0) },\quad t\geq t_0,$$
where $\sigma_0=\sigma -K_1(h)e^{\sigma q}.$ Therefore we need $h$ small enough such that
$$\sigma >K_1(h)e^{\sigma q}=a_0^2\left[2h+w(r;h)\right]K e^{\sigma q}.$$

On the other hand, the necessary condition to use Halanay type inequality is: $h$ small enough such that
$$\sigma >K_1(h)=a_0^2\left[2h+w(r;h)\right]K,$$
then 
$$|E_h(t)|\leq \left[ K\norm{E_{h_{t_0}}}_q +K_1(h)M_1\norm{\varphi}_qe^{\sigma t_0}t\right]e^{-\eta(t-t_0) },\quad t\geq t_0,$$
where $\eta$ is the positive real  solution of 
\begin{equation}
\eta=\sigma-K_1(h)e^{\eta q}.\nonumber
\end{equation}
We note that the size of $h$ is independent of the delay size $q$, and the number $-\eta$ is the unique real solution of the characteristic equation 
\begin{equation}\label{eq: chareq} 
\lambda=-\sigma +K_1(h)e^{-\lambda q},
\end{equation}
corresponding to the differential equations with delay
$$y'(t)=-\sigma y(t)+K_1(h)y(t-q).$$
In fact $-\eta$ is the eigenvalue of the characteristic equation \eqref{eq: chareq} with the greatest real part.  

\medskip

Now we can prove that the solutions of \eqref{eq: ddvd}-\eqref{eq: icdd} approximate uniformly the solutions of \eqref{eq: devd}-\eqref{eq: icd}, and also that
zero solution of \eqref{eq: ddvd}-\eqref{eq: icdd} is uniformly asymptotically stable.

\begin{cor}\label{cor: 2}
Under the conditions of the Theorem \ref{teo: 1}, we have that:
\begin{enumerate}
\item $|x(t)-z_{h}\left([t]_h\right)|  \to 0 \mbox{ as }h\to 0,\mbox{ for }t>0;  $
\item the zero solution  of the difference equations with delay
 \eqref{eq: ddvd}-\eqref{eq: icdd} is uniform asymptotically stable.
\end{enumerate}
\end{cor}

\begin{proof}
In section \ref{sec: DPCA} we have shown that \eqref{eq: ddvd}-\eqref{eq: icdd}  correspond to a discrete version of differential equation
\eqref{eq: devd}-\eqref{eq: icd}. We recall that $z_{h}(h n )=\z_{h}(n)$ for $n$ any positive integer.
If $t<t_0$ we use Theorem \ref{Teo2.1.1}. If $t\geq t_0$  then
\begin{eqnarray*}
|x(t)-z_{h}\left([t]_h\right)|&\leq &|x(t)-x\left([t]_h\right)|+|x\left([t]_h\right)-z_{h}\left([t]_h\right)|.
\end{eqnarray*}
Then, from inequality \eqref{eq: 2.2.3}, we have 
\begin{eqnarray*}
|x(t)-z_{h}\left([t]_h\right)|&\leq &|x(t)-x\left([t]_h\right)|+\left[K\norm{E_{h_{t_0}}}_q+K_1(h)M_1\norm{\varphi}_q [t]_h\right]e^{-\eta([t]_h-t_0) }.
\end{eqnarray*}
Next, for $\eps>0$ there are positive constants $h_1, h_2$ and $h_3$ such that:\\
If $h<h_1$ then  $|x(t)-x\left([t]_h\right)|<\frac{\eps}{3}$, since continuity of $x$. If $h<h_2$ then  $\norm{E_{h_{t_0}}}_q<\frac{\eps}{3K}$, from Theorem \ref{Teo2.1.1}. If $h<h_3$ then  $K_1(h)<\frac{\eps}{3M_1\norm{\varphi}_q}\frac{\eta e}{e^{\eta t_0}}$. Therefore, for $h<\min\{h_1,h_2,h_3\}$ it is follows
\begin{eqnarray*}
|x(t)-z_{h}\left([t]_h\right)|&\leq &\frac{\eps}{3}+\frac{\eps}{3}+\frac{\eps}{3}=\eps.
\end{eqnarray*}
We have shown part 1. Since inequality \eqref{eq: 2.2.3} we have also
\begin{eqnarray*}
|x(nh)-\z_{h}(n)| & \le & \left[K\norm{E_{h_{t_0}}}_q+K_1(h)M_1\norm{\varphi}_q nh \right]e^{-\eta(nh-t_0) } \to 0\mbox{ as }n\to \infty. 
 \end{eqnarray*}
Since the zero solution of \eqref{eq: devd}-\eqref{eq: icd} is uniformly asymptotically stable, and 
$\norm{x(nh)-\z_{h}(n)}$ decay rate exponentially to zero, it is follows that  
$\z_{h}(n)$ tends to zero, for every initial conditions $\varphi$. 
Therefore the  zero solution of \eqref{eq: ddvd}-\eqref{eq: icdd} is uniformly asymptotically stable too.
\end{proof}

Thus we have shown that under the hypotheses of Theorem \ref{teo: 1} the numerical approximations
of equation \eqref{eq: ddvd} are good for all $ T>0 $, independent of the size of $ T $. 
Moreover the corresponding discrete 
difference equation is  uniformly asymptotically stable also. We use piecewise constant argument, 
theory of functional differential equations and Halanay-type inequality  in the proof.
Furthermore, our result is independent of delay size, this was possible 
thanks to our use of inequality Halanay. Thus we extend and improve the results of \cite{cooke1994}.

\section{Applications and examples}\label{Aae}

In our result we assume that the zero solution of \eqref{eq: devd} is uniformly asymptotically stable,
the problem of find necessary condition for uniform stability of non-autonomous differential equations 
with variable delay called the attention of several authors because its difficulty.
% For example,if there exists $\alpha>0$ such that $a(t)\leq \alpha$ and $\alpha q\leq 3/2$, then the zero solution of \eqref{eq: devd} is uniformly stable, but if $\alpha q >3/2$ then there are equations with unbounded solutions (see \cite{gopalsamy2013,yoneyama1987}).
Next we recall some stability criteria for equation \eqref{eq: devd} and apply our result to obtain stability criteria for difference equation.  

A classic result of stability for functional differential equations can be found in \cite{yorke1970}.
A consequence of Yorke's theorem is:
\begin{Athm}[Yorke]\label{teo: yorke}
If the function $a(\cdot)$ satisfy 
$$0< a(t)\leq \alpha, \quad t\geq 0;$$
for a positive constant $\alpha$ such that
$$ 0 <\alpha q<\frac{3}{2}. $$
Then the zero solution of \eqref{eq: devd} is uniformly asymptotically stable.
\end{Athm}

\begin{exam}\label{ej: 2.1}
We consider the non-autonomous linear differential equations with delay 
\begin{equation}\label{eq: ex2.1}
x'(t)=-\left[1+\frac{\sin(t)}{3}\right]x\left(t-|\cos(t)|\right).
\end{equation}
Since $0 <1+\frac{\sin(t)}{3}<\frac{4}{3}$  and $0\leq |\cos(t)|\leq 1$, it follows that $\alpha q=\frac{4}{3}<\frac{3}{2}$, therefore from Theorem A the zero solution of \eqref{eq: ex2.1} is uniformly asymptotically stable so (A1) and (A2) holds. Since the function $\cos(x)$ is uniformly continuous on $\R$ (A3) holds.
So Theorem \ref{teo: 1} and Corollary \ref{cor: 2} are valid. Therefore we can approximate the solution of \eqref{eq: ex2.1} by the family of difference equations \eqref{eq: ddvd}  corresponding to \eqref{eq: ex2.1}
\begin{equation}\label{eq: discrete}
 \z_h(n+1)=\z_h(n)-\mathfrak{a}_h(n)\z_h(n-k_n),\nonumber
\end{equation}
where 
$$\mathfrak{a}_h(n)=\int_{nh}^{(n+1)h}\left[1+\frac{\sin(s)}{3}\right]ds=h-\frac{\cos((n+1)h)-\cos(nh)}{3},$$
and $$k_n=\Big[\frac{|\cos(nh)|}{h}\Big].$$
It follows that the zero solution of
\begin{equation}\label{eq: discrete}
 \z_h(n+1)=\z_h(n)-\left\{h-\frac{\cos((n+1)h)-\cos(nh)}{3}\right\}\z_h\left(n-\Big[\frac{|\cos(nh)|}{h}\Big]\right),
\end{equation}
is uniformly asymptotically stable.
We note that, since mean value theorem, \eqref{eq: discrete} is equivalent to
\begin{eqnarray*}
\z_h(n+1)-\z_h(n)&=&-\left\{h-\frac{\cos((n+1)h)-\cos(nh)}{3}\right\}\z_h\left(n-\Big[\frac{|\cos(nh)|}{h}\Big]\right)\\
&=&-h\left\{1-\frac{1}{3}\frac{\cos((n+1)h)-\cos(nh)}{h}\right\}\z_h\left(n-\Big[\frac{|\cos(nh)|}{h}\Big]\right)\\
&=&-h\left\{1+\frac{\sin(c_{n+1})}{3}\right\}\z_h\left(n-\Big[\frac{|\cos(nh)|}{h}\Big]\right),
\end{eqnarray*} 
for some $c_{n+1}\in (nh,nh+h)$.
\begin{eqnarray}
\frac{\z_h(n+1)-\z_h(n)}{h}&=&-\left[1+\frac{\sin(c_{n+1})}{3}\right]\z_h\left(n-\Big[\frac{|\cos(nh)|}{h}\Big]\right).
\end{eqnarray}
\end{exam}

\begin{figure}[hbtp]
\caption{Approximation solution of \eqref{ej: 2.1}
with initial function $\phi\equiv 5$  with $h=0.5$}
\centering
\includegraphics[scale=.6]{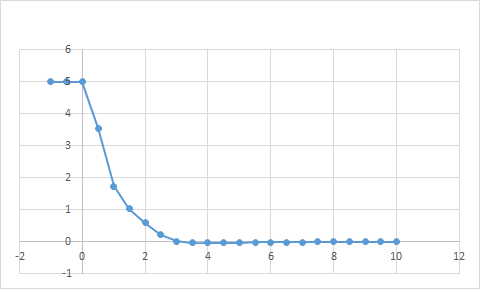}
\end{figure}
\begin{figure}[hbtp]
\caption{Approximation solution of \eqref{ej: 2.1}
with initial function $\phi\equiv 5$  with  $h=0.\overline{3}$}
\centering
\includegraphics[scale=.6]{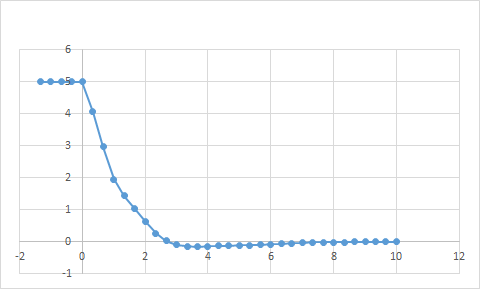}
\end{figure}
\begin{figure}[hbtp]
\caption{Approximation solution of \eqref{ej: 2.1}
with initial function $\phi\equiv 5$  with  $h=0.25$}
\centering
\includegraphics[scale=.6]{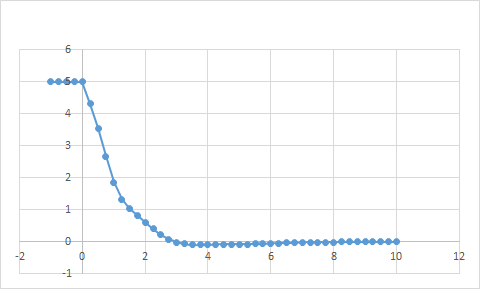}
\end{figure}

\bigskip 

\section[Conclusion]{Conclusion}

In order to get our overall goal we need identify the main assumptions, techniques and methods
used in papers about approximation and transference of stability properties between 
solutions of differential equations with delay and the corresponding difference equations.

About the techniques and methods used, we note that the results of \cite{cooke1994} and \cite{gyori2002} rest
on the functional differential  equations that error $E_h(t)$ satisfy. Actually $E_h(t)$ satisfy a linear 
non-homogeneous or semi-linear  functional differential equations. If the linear 
homogeneous  differential equations with delay is uniformly asymptotically stable, then is possible conclude 
the exponential decay rate of the error by using an integral inequality. In the previous work they consider 
Gronwall-Bellman inequality, however we use Halanay inequality because is more appropriate for delay 
differential equations. We also note that \cite{mohamad2000} and \cite{liz2002} used Halanay inequality as key 
technique to prove that the zero solution of both continuous functional differential equations and discrete
equations are exponentially stable. However, this technique does not state any estimation of the error. 

Our Theorem \ref{teo: 1} state the approximation and transference of exponential stability
properties of solutions of non-autonomous differential equations with variable delay and retarded functional 
differential equation with feedback, respectively,  to the corresponding difference equation 
by using piecewise constant argument, the techniques and methods above mentioned.

Natural applications of our results can be found in the systems of differential equations used to model cellular 
neural networks and identification of  parameters in functional differential equations see, for instance, \cite{mohamad2003,abbas2013} and \cite{hartungturi1997,hartung1998,hartung2000} respectively.

\section*{Acknowledgements}
I thank Prof. Dr. Manuel Pinto for initiate me into mathematical research and suggest me this subject. I also appreciate the support of CIR 1418 research project at the Universidad Central de Chile.
\bibliography{referencias}

\bibliographystyle{apalike}

\end{document}